\documentclass[12pt]{amsart}
\usepackage[utf8]{inputenc}

\usepackage{extsizes}
\usepackage{amssymb, amsmath}
\usepackage{fourier}
\usepackage{color}
\usepackage{hyperref}

\usepackage{geometry}
\usepackage{tikz}
\usetikzlibrary{positioning, shapes.geometric, arrows.meta}
\usepackage{amssymb}
\usepackage{amsmath}
\usepackage{amsthm}
\usepackage{amsfonts}
\usepackage{bbm}
\usepackage{enumitem} 
\usepackage{pgfplots}
\pgfplotsset{compat=1.18} 
\usepackage{booktabs}
\usepackage{graphicx}
\usepackage[T1]{fontenc}
\usepackage{doi}
\usepackage{float} 
\usepackage{xcolor}

\usepackage{bookmark}
\usepackage{hyperref}
\hypersetup{pdfstartview={FitH}}

\newcommand{\C}{\mathbb{C}}
\newcommand{\norm}[1]{\lVert #1 \rVert}
\newcommand{\abs}[1]{| #1 |}

\newcommand{\A}{\mathcal A}
\newcommand{\I}{\mathbf{1}}
\newcommand{\pos}{\A^{+}}
\newcommand{\Apos}{\A^{++}}
\newcommand{\Ap}{\A^{++}} 
\newcommand{\St}{S(\A)}
\newcommand{\Geomean}{\mathbin{\#}}

\newcommand{\Mkpp}{{\mathbb M_2}(\mathbb C)^{++}}

\newcommand{\sigone}{\sigma_{ch}}
\newcommand{\sigtwo}{\sigma_{kah}}
\newcommand{\sigthree}{\sigma_w}

\newcommand{\App}{\mathcal A^{++}}

\newtheorem{thm}{Theorem}[section]
\newtheorem{lem}[thm]{Lemma}
\newtheorem{prop}[thm]{Proposition}

\theoremstyle{definition}

\newtheorem{remark}[thm]{Remark}

\numberwithin{equation}{section}

\begin{document}
	
	\title[]
	{On interrelations among different versions of a Heron type mean and commutativity in $C^*$-Algebras}
	
	\author[L.~Moln\'ar]{Lajos Moln\'ar}
	
	\address{Bolyai Institute, University of Szeged, and HUN-REN–SZTE Analysis and Applications Research Group, Aradi v\'ertan\'uk tere 1.,
H-6720 Szeged}
	
	\email{molnarl@math.u-szeged.hu}
	
	\urladdr{http://www.math.u-szeged.hu/~molnarl}

	\author[T.~Zhang]{Teng Zhang}
	
	\address{School of Mathematics and Statistics, Xi'an Jiaotong University, Xi'an 710049, P. R. China}
		
	\email{teng.zhang@stu.xjtu.edu.cn}


\thanks{The first author was supported by the National Research, Development and Innovation Office of
Hungary, NKFIH, Grant No. $\text{ADVANCED}\_150059$. 
The second author was supported by the China Scholarship Council, the Young Elite Scientists Sponsorship Program for PhD Students (China Association for Science and Technology), and the Fundamental Research Funds for the Central Universities at Xi'an Jiaotong University (Grant No.~xzy022024045). 
}

	\begin{abstract}
The extension of the concept of a mean of positive real numbers to noncommutative settings, e.g., for Hilbert space operators, is a widely studied question. For example, in quantum information science, it is an important issue to find such extensions that fit the "best" to the studied physical problems. In fact, typically, there are many different ways of extension which, for commuting variables, all give the same value. In this paper, we are concerned with the converse: to what extent the coincidence of two extensions determines commutativity.
    Concretely, in our present work, we consider three different versions of the most common Heron type mean on the positive definite cone of a $C^*$-algebra: the Kubo-Ando type Heron mean, the naive or conventional version of the Heron mean, and the Wasserstein mean. We study equality relations among those objects and verify that they are closely connected to certain commutativity properties. They characterize either the commutativity of particular pairs of elements of a positive definite cone, or the centrality of positive definite elements, or the commutativity of the underlying algebra. 
	\end{abstract}
	
	\maketitle
    
	\section{Introduction}

	Throughout the paper, by a $C^*$-algebra we always mean a unital $C^*$-algebra with identity $\I$.
    Let $\A$ be such an algebra. We denote by $\A_s$ the linear space of all self-adjoint elements of $\A$, by $\A^{+}$ the set of all positive (in other words, positive semidefinite) elements of $\A$ (which are those self-adjoint elements that have nonnegative spectrum), and by $\App$ the set of all positive invertible (or, in other words, positive definite) elements of $\A$.
	We call $\A^+$ the positive semidefinite cone, and $\App$ the positive definite cone of $\A$. For any Hilbert space $H$, the $C^*$-algebra of all bounded linear operators on $H$ is denoted by $B(H)$.

    There has been, and there still is, extensive research concerning matrix and operator means defined on positive cones of matrices or operators. Let us just refer to the beautiful Kubo-Ando theory, where operator means are parameterized by operator monotone functions \cite{KA80}. Recently, much attention has also been paid to certain non Kubo-Ando type means. In this paper, we particularly deal with different versions (Kubo-Ando type, and non Kubo-Ando types) of a Heron mean.

    Recall that in the special case of positive real numbers, any affine combination of the arithmetic and geometric means is called a Heron mean. Here, we are interested in the most natural choice, which is half of the arithmetic mean plus half of the geometric mean. Part of the motivation for the present investigation comes from the current special interest in the Wasserstein mean (the definition is given below). To list some recent papers relevant to the study of that mean, we mention \cite{BJL,BJL2,1, 2,3,7,5,4,6}.
    
    Below, we consider the means in the general setting of $C^*$-algebras.
    Let us begin with the general notion of Kubo-Ando means. If $f$ is a positive operator monotone function on the set of positive real numbers, then the corresponding Kubo-Ando mean on the positive definite cone $\App$ is the two-variable transformation
    \begin{equation*}
        (A,B)\mapsto A^{1/2}f(A^{-1/2} BA^{-1/2})A^{1/2}
    \end{equation*}
    which, of course, can also be viewed as a particular binary operation on $\App$. If $f$ is the square-root function, then we obtain the Kubo-Ando type geometric mean: for any $A,B\in\Apos$, we set
	\begin{equation}\label{E:6}
	A\Geomean B = A^{1/2}\bigl(A^{-1/2}BA^{-1/2}\bigr)^{1/2}A^{1/2}.
	\end{equation}
    With this at hand, the Wasserstein mean is defined by
    \begin{equation}\label{eq:sig3}
    A\sigthree B= \frac14({A+B + A(A^{-1}\Geomean B) + (A^{-1}\Geomean B)A)}, \quad A,B\in \App.
    \end{equation}
    This mean, for matrices, was introduced in \cite{BJL}. Let us recall some fundamental properties of $\Geomean, \sigthree$ in that particular setting.
    Most notably, those means are connected to Riemannian metrics. Indeed, the Kubo-Ando type geometric mean $A\Geomean B$ is well-known to be the midpoint of the curve connecting $A$ and $B$ that is the geodesic in the so-called affine invariant Riemannian metric. A similar statement is true for the Wasserstein mean in another Riemannian structure  on the positive definite cone of matrices, where the geodesic distance is the so-called Bures metric (or Bures-Wasserstein metric), which plays an important role in the theory of quantum information, \cite{BJL, BJL2}.

    Apparently, for positive numbers, the Wasserstein mean coincides with the above mentioned Heron mean, which has two other straightforward extensions to the noncommutative setting.
    The first variant is the Kubo-Ando type mean  
    \begin{align*}
		A\sigtwo B
		=\frac14\bigl(A+B+2(A\Geomean B)\bigr)= A^{1/2}\left(\frac{\I+(A^{-1/2}BA^{-1/2})^{1/2}}{2}\right)^2A^{1/2}, \quad A,B\in \App
	\end{align*}
    which corresponds to the operator monotone function $t\mapsto ((1+\sqrt{t})/2)^2$ and, in matrix theory, it is usually called matrix Heron mean. (Actually, $\sigtwo$ is the Kubo-Ando power mean corresponding to the parameter value $p=1/2$.) The other one is the most naive and conventional variant, which is
	\begin{align*}
		A\sigone B= \left(\frac{A^{1/2}+B^{1/2}}{2}\right)^2, \quad A,B\in \App
		.
	\end{align*}
    Clearly, for commuting $A,B$, we have
    $A\sigtwo B=A\sigone B= A\sigthree B$.
    
    We point out that in the former paper \cite{MolnarSimon}, the authors investigated the above introduced Heron type means and provided several results that are mainly related to their algebraic properties (recall again that means are a sort of binary operations on positive definite cones). They also elaborated on interrelations of those means and formulated some open problems that we solve here. 
    
    The common feature of the results presented below is that the considered interrelations are equality relations among the means in question, and they are proved to be very closely connected to commutativity properties. Indeed, those relations characterize either the commutativity of given pairs of elements, or characterize the centrality of given elements of a positive definite cone, or characterize the commutativity of the underlying $C^*$-algebra.

    It is an essential novelty in our current investigations that we bring in a notable tool from the theory of $C^*$-algebras called the excision of pure states.

    \section{Equalities for a fixed pair of elements}

In Proposition 3 in \cite{MolnarSimon}, it was proved that if, in the $C^*$-algebra $\A$, for a given pair $A,B\in \App$ we have either
\begin{equation*}
A\sigtwo B=A\sigone B
\end{equation*} or 
\begin{equation*}
A\sigtwo B=A\sigthree B,
\end{equation*}
then $A,B$ necessarily commute. The natural question arises: what happens, in the same respect, to the means $\sigone, \sigthree$? This problem, which was raised as Problem 1 in \cite{MolnarSimon}, is more difficult, but the conclusion is the same as in the previous two cases, as we will see below.
In the proof, we make use of the following nice and highly nontrivial result due to Ando and Hayashi \cite{AH07}.

	\begin{thm}\label{lem:AH}
		Let $H$ be a Hilbert space and let $X,Y\in B(H)$. If
		\[
		\abs{X+Y}=\abs{X}+\abs{Y},
		\]
		then there exists a partial isometry $U\in B(H)$ such that
		\[
		X=U\abs{X}\qquad\text{and}\qquad Y=U\abs{Y}.
		\]
	\end{thm}

Our first result is formulated as follows.

	\begin{thm}\label{thm:prob1}
		Let $\A$ be a $C^*$-algebra and $A,B\in\Apos$.
		If 
        \begin{equation*}
        A\sigone B =A\sigthree B, 
        \end{equation*}
        then $AB=BA$.
	\end{thm}

As the converse statement, i.e., that $AB=BA$ implies $A\sigone B =A\sigthree B$ is obviously true, this result actually characterizes the commuting elements of the positive definite cone $\App$. (Clearly, the same holds in relation to the above mentioned Proposition 3 in \cite{MolnarSimon}.)
    
	\begin{proof}
    We begin by mentioning that, by the Gelfand-Naimark theorem, $\A$ is isometrically *-isomorphic to a closed *-subalgebra of some $B(H)$ containing the identity operator $I$. Hence, it is sufficient to verify our statement only for such an algebra of Hilbert space operators, that is, only in relation with arbitrary invertible positive Hilbert space operators $A,B\in B(H)$.
		Set
		\[
		X=\big(A^{1/2}BA^{1/2}\big)^{1/2},\quad Y=B^{1/2}A^{1/2}.
		\]
		Since $A^{-1}\Geomean B=A^{-1/2}XA^{-1/2}$, the equality $A\sigone B=A\sigthree B$
		is equivalent to
		\begin{equation*}\label{eq:cross-new}
			A^{1/2}B^{1/2}+B^{1/2}A^{1/2}
			= A^{1/2}XA^{-1/2}+A^{-1/2}XA^{1/2}.
		\end{equation*}
		Multiplying this equality from the left and right by $A^{1/2}$, we obtain
		\begin{equation}\label{eq:cross-new2}
			AY+Y^*A=AX+XA.
		\end{equation}
Next, note that
		\[
		Y^*Y=A^{1/2}BA^{1/2}=X^2,
		\enskip\text{hence}\enskip \abs{Y}=(Y^*Y)^{1/2}=X.
		\]
Using \eqref{eq:cross-new2}, we compute
		\begin{align*}
			(A+Y)^*(A+Y)
			&=A^2+Y^*Y+AY+Y^*A\\
			&=A^2+X^2+AX+XA=(A+X)^2.
		\end{align*}
		Taking the positive square-root yields
		\begin{equation*}\label{eq:triangle}
			\abs{A+Y}=A+X=\abs{A}+\abs{Y}.
		\end{equation*}
        Since $A=\abs{A}$ is invertible, by Theorem~\ref{lem:AH} it follows that $Y=\abs{Y}$, from which we can infer that $Y$ is self-adjoint and then that $A,B$ commute. 
	\end{proof}

So, we know that if, for a given pair of positive definite elements, two of the three versions of the Heron mean that we consider in the paper are equal, then the elements in question need to commute. 

If the $C^*$-algebra $\A$ carries a faithful trace, even much more is true.

By a trace on $\A$, we mean a positive linear functional $\tau:\A\to \C$ which satisfies $\tau(XY)=\tau(YX)$, $X,Y\in \A$.
It is called faithful if for any $A\in \A^+$, the equality $\tau(A)=0$ implies $A=0$. 

If $\tau$ is a trace on $\A$, then the following inequalities hold.
	\begin{equation}\label{E:44}
       \tau(A\sigtwo B)\leq \tau(A \sigone B)\leq \tau(A\sigthree B), \quad A,B\in \App,
   \end{equation}
see \cite{MolnarSimon}, p. 219.
In Proposition 4 in \cite{MolnarSimon}, we proved that for any faithful trace $\tau$ on $\A$, if $A,B\in \App$ are given and satisfy
	either 
    \[
    \tau(A\sigone B)=\tau(A\sigtwo B)
    \]
    or 
    \[
    \tau(A \sigone B)= \tau(A\sigthree B), 
    \]
    then $A,B$ necessarily commute. Since the converse statements are obviously true, that gives a characterization of commuting elements of $\App$ in terms of single numerical equalities involving a faithful trace.

\begin{remark}
For curiosity, let us remark that the inequalities appearing in \eqref{E:44} can be used to characterize traces. We mean that using the method developed in \cite{HM}, one can show the following. If $\varphi$ is a positive linear functional on the $C^*$-algebra $\A$, then it is a trace if and only if, for example, the inequality 
\begin{equation*}\label{E:5}
       \varphi(A\sigtwo B)\leq \varphi(A\sigthree B), \quad A,B\in \App
   \end{equation*}
holds. We do not present the details here (it does not belong to the topic of the present paper), only mention this hopefully interesting fact to complement the inequalities in \eqref{E:44}.
\end{remark}

    \section{Equalities in the $C^*$-norm}\label{S:3}

Of course, most of the $C^*$-algebras do not carry nontrivial traces. Hence, motivated by the last part of the previous section, one can raise the question of what happens if we consider the $C^*$-norm in place of the trace norm above?
Obviously, one cannot expect such a conclusion as the ones concerning traces. Just think of the difference in the amount of information that the spectral norm and the trace of positive matrices contain.

    Let us next mention that it was pointed out in \cite{MolnarSimon} (see page 220 there) that we have the following inequalities
    \[
    \| A\sigtwo X\| \leq \|A\sigone X\|\leq \|A\sigthree X\|, \quad A,X\in \App.
    \]
    In Proposition 6 of that paper, it was shown in the specific case of the full operator algebra $B(H)$ over a Hilbert space $H$ that, in either of the above two inequalities, we have equality for a given $A$ and for all $X$ if and only if $A$ is a positive scalar multiple of the identity. The problem was also raised there (see Problem 2 in \cite{MolnarSimon}) whether, in the context of a general $C^*$-algebra, the same equalities characterize the centrality of the given element $A$. It turns out that the problem is surprisingly difficult; in the solution, which takes several pages, we need to employ various highly nontrivial tools, among others, the excision of pure states of $C^*$-algebras. 

The result reads as follows.
    
	\begin{thm}\label{thm:main}
		Let $\A$ be a $C^*$-algebra and $A\in\Ap$.
		Assume that we have either
		\[
		\norm{A\sigone X}=\norm{A\sigtwo X}, \quad X\in\Ap,
		\]
		or
		\[
		\norm{A\sigone X}=\norm{A\sigthree X},\quad X\in\Ap.
		\]
		Then $A$ is a central element in $\A$. (The converse statements hold trivially.) 
	\end{thm} 

The argument to verify the theorem rests, basically, on the next proposition. Let us recall the following. A state of $\A$ is a positive linear functional $\varphi:\A\to\C$ with $\varphi(\I)=1$. We denote the state space of $\A$, i.e., the convex set of all states of $\A$, by $\St$. It is a weak*-compact subset of the unit ball of $\A^*$. The extreme points of $\St$ are called pure states, and it is well-known that a state is pure exactly when the corresponding GNS construction yields an irreducible representation.

The crucial observation is the following.

\begin{prop}\label{prop:pure-identities}
		Given an $A\in\Ap$, the following hold true.
		\begin{enumerate}[label=(\alph*)]
			\item If $\norm{A\sigone X}=\norm{A\sigtwo X}$ for all $X\in\Ap$, then every pure state $\psi$ on $\A$ satisfies
			\begin{equation}\label{eq:pure2}
				\psi(A^{1/2})=\frac{1}{\sqrt{\psi(A^{-1})}}.
			\end{equation}
			\item If $\norm{A\sigone X}=\norm{A\sigthree X}$ for all $X\in\Ap$, then every pure state $\psi$ on $\A$ satisfies
			\begin{equation}\label{eq:pure3}
				\psi(A^{1/2})=\sqrt{\psi(A)}.
			\end{equation}
		\end{enumerate}
	\end{prop}

We will prove this proposition later, the proof has several ingredients. Supposing this proposition is proved, to the proof of Theorem \ref{thm:main}, we will need the next assertion, whose proof we can present right away.

\begin{prop}\label{P:1}
Assume that $f:(0,\infty) \to \mathbb R$ is a continuous function, which is not affine on any nontrivial subinterval of $(0,\infty)$. If $A\in \App$ is such that for every pure state $\psi$ on $\A$ we have
\[
\psi(f(A))=f(\psi(A)),
\]
then $A$ is a central element of $\A$.
\end{prop}

To prove this proposition, we recall the following closely related statement that appeared as Lemma 7 in \cite{MolnarSimon}.

\begin{lem}\label{lem:MS-L7}
		Let $H$ be a Hilbert space, and let $f:(0,\infty)\to\mathbb R$ be a continuous function that is not affine on any nontrivial subinterval of $(0,\infty)$.
		If $T\in B(H)^{++}$ is such that for every unit vector $x\in H$ we have
		\[
		\langle f(T)x,x\rangle=f(\langle Tx,x\rangle),
		\]
		then $T$ is a positive scalar multiple of the identity.
	\end{lem}

\begin{proof}[Proof of Proposition \ref{P:1}]
		Let $\pi:\A\to B(H)$ be any irreducible $^*$-representation and
		$x\in H$ any unit vector. Then the vector state
		\[
		\omega_x(A)=\langle \pi(A)x,x\rangle, \quad A\in \A
		\]
		is known to be a pure state on $\A$. Indeed, we have that $x$ is necessarily a cyclic vector of $\pi$ and the GNS construction corresponding to $\omega_x$ is unitarily equivalent to $\pi$, and hence is irreducible. This verifies that the state $\omega_x$ is pure.

        After this, by the conditions in the proposition, it follows that 
        \[
        \langle f(\pi(A))x,x\rangle=\langle \pi(f(A))x,x\rangle=\omega_x(f(A))=f(\omega_x(A))=f(\langle \pi(A)x,x\rangle).
        \]
        Applying Lemma \ref{lem:MS-L7}, we obtain that the operator $\pi(A)$ is necessarily a scalar multiple of the identity. Therefore, it follows that $\pi(AX-XA)=\pi(A)\pi(X)-\pi(X)\pi(A)=0$ holds for any $X\in \A$ and irreducible representation $\pi$ of $\A$. This gives us that $AX-XA=0$ for any $X\in \A$, i.e., $A$ is a central element of $\A$.
\end{proof}
    
Apparently, Proposition \ref{prop:pure-identities} and Proposition \ref{P:1} together will result in Theorem \ref{thm:main}.
So, our remaining task is to verify Proposition \ref{prop:pure-identities}. Its proof rests on some other statements. Before presenting them, we recall some facts and introduce a notation.

    For any $T\in \A^+$, we have
	\begin{equation}\label{eq:norm-by-states}
		\norm{T}=\sup_{\varphi\in\St}\varphi(T).
	\end{equation}
	Since $\St$ is weak$^*$-compact and the map $\varphi\mapsto\varphi(T)$ is weak$^*$-continuous,
	the supremum in \eqref{eq:norm-by-states} is actually attained. Next, for an arbitrary $Y\in\pos$ with $\norm{Y}=1$, we define
	\[
	M_Y=\{\varphi\in\St:\ \varphi(Y)=1\}.
	\]
	The set $M_Y$ is nonempty and weak$^*$-compact.

To prove Proposition \ref{prop:pure-identities}, we verify the following asymptotic formulas.
	
	\begin{lem}\label{lem:asymptotic}
		Let $Y\in\pos$ be with $\norm{Y}=1$, let $B\in\A_s$, and let $C\in\pos$.
		Set
		\[
		m=\sup_{\varphi\in M_Y}\varphi(B)(\in[-\norm{B},\norm{B}]).
		\]
		\begin{enumerate}[label=(\alph*)]
			\item Assume that there exists $t_0\ge 0$ such that $B+tY\in\pos$ for all $t\ge t_0$.
			Then we have
			\begin{equation}\label{eq:lim1}
				\lim_{t\to\infty}\bigl(\norm{B+tY}-t\bigr)=m.
			\end{equation}
			\item Assume that there exists $t_1\ge 0$ such that $C+t^2Y+tB\in\pos$ for all $t\ge t_1$.
			Then we have
			\begin{equation}\label{eq:lim2}
				\lim_{t\to\infty}\frac{\norm{C+t^2Y+tB}-t^2}{t}=m.
			\end{equation}
		\end{enumerate}
	\end{lem}

    \begin{proof}
		(a) For $t\ge t_0$, apply \eqref{eq:norm-by-states} to the positive element $B+tY$ and obtain
		\[
		\norm{B+tY}=\sup_{\varphi\in\St}\bigl(\varphi(B)+t\varphi(Y)\bigr),
		\]
		Therefore,
		\[
		\norm{B+tY}-t=\sup_{\varphi\in\St}\bigl(\varphi(B)+t(\varphi(Y)-1)\bigr)=:f(t), \quad t\geq t_0.
		\]
        On the one hand,  if $\varphi\in M_Y$, then $f(t)\ge \varphi(B)$ for all $t\ge t_0$, so
		$\liminf_{t\to\infty}f(t)\ge m$.
		
		To obtain a related inequality in the reverse direction, first fix $\varepsilon>0$ and set
		\[
		K_\varepsilon=\{\varphi\in\St:\varphi(Y)\ge 1-\varepsilon\},
		\]
		which is a weak$^*$-compact set. Denote
		\[
		m_\varepsilon=\sup_{\varphi\in K_\varepsilon}\varphi(B).
		\]
		Note that as $\varepsilon\downarrow 0$, the sets $K_\varepsilon$ decrease to $M_Y$, hence
		$m_\varepsilon$ is a decreasing net in $[-\norm{B},\norm{B}]$ and therefore it has a limit.
        
		If $\varphi\in \St \setminus K_\varepsilon$, then $\varphi(Y)< 1-\varepsilon$ and thus
		\[
		\varphi(B )+t(\varphi(Y)-1)\le \norm{B}-\varepsilon t.
		\]
		Since $m_\varepsilon\ge -\|B\|$, choosing $t\ge 2\|B\|/\varepsilon$ ensures
		\[
		\norm{B}-\varepsilon t\le -\|B\|\le m_\varepsilon.
		\]
		Therefore, for all $t\ge \max\{t_0,2\|B\|/\varepsilon\}$, the inequality
        \[
		\varphi(B)+t(\varphi(Y)-1)\le m_\varepsilon
		\]
        holds for any $\varphi\in \St \setminus K_\varepsilon$ as well as for any $\varphi\in K_\varepsilon$, hence we have $
		f(t)\le m_\varepsilon$ and
		consequently $\limsup_{t\to\infty}f(t)\le m_\varepsilon$.
		
		Finally, since $K_\varepsilon\downarrow M_y$ and each $K_\varepsilon$ is weak$^*$-compact, we can choose
		$\varphi_\varepsilon\in K_\varepsilon$ with $\varphi_\varepsilon(B)=m_\varepsilon$ and take a weak$^*$-convergent subnet of
		$(\varphi_\varepsilon)$ converging to some $\varphi\in\St$. We necessarily have $\varphi\in M_Y$. Therefore,
		the corresponding subnet of $(m_\varepsilon)$ converges to $\varphi(B)\le m$. Since $m\le m_\varepsilon$ holds for all $\varepsilon>0$, we conclude $m_\varepsilon\downarrow m$,
		so $\limsup_{t\to\infty}f(t)\le m$. This, together with $\liminf_{t\to\infty}f(t)\geq m$ proves \eqref{eq:lim1}.
		
		(b) For $t\ge t_1$, apply \eqref{eq:norm-by-states} to $C+t^2Y+tB\in\pos$ to obtain
		\[
		\frac{\norm{C+t^2Y+tB}-t^2}{t}
		=\sup_{\varphi\in\St}\left(\varphi(B)+t(\varphi(Y)-1)+\frac{\varphi(C)}{t}\right)=:g(t).
		\]
		On the one hand, if $\varphi\in M_Y$, then $g(t)\ge \varphi(B)+\frac{\varphi(C)}{t}$, hence
		$\liminf_{t\to\infty}g(t)\ge m$.
		
        Next, fix $\delta>0$ and set
		\[
		K_\delta=\{\varphi\in\St:\varphi(Y)\ge 1-\delta\},\quad
		m_\delta=\sup_{\varphi\in K_\delta}\varphi(B).
		\]
		Choose $\varphi_0\in M_Y$ with $\varphi_0(B)\ge m-\delta$.
		Then for all $t\geq t_1$, we have
		\[
		g(t)\ge \varphi_0(B)+\frac{\varphi_0(C)}{t}\ge m-\delta-\frac{\norm{C}}{t}.
		\]
		If $\varphi\in \St\setminus K_\delta$, then $\varphi(Y)< 1-\delta$ and
		\[
		\varphi(B)+t(\varphi(Y)-1)+\frac{\varphi(C)}{t}
		\le \norm{B}-\delta t+\frac{\norm{C}}{t}.
		\]
		Since the right-hand side tends to $-\infty$ as $t\to\infty$, for all sufficiently large $t$,
		the supremum defining $g(t)$ is attained in $K_\delta$, and hence
		\[
		g(t)\le m_\delta+\frac{\norm{C}}{t}.
		\]
		Thus, $\limsup_{t\to\infty}g(t)\le m_\delta$. As $\delta\downarrow 0$, the same compactness argument as in the proof of (a)
		gives $m_\delta\downarrow m$, and we obtain \eqref{eq:lim2}.
	\end{proof}

To what comes next, we note the following. Clearly, the definition \eqref{E:6} of the Kubo-Ando geometric mean can be straightforwardly extended for any positive (semidefinite) $B$ and then the same holds for the Wasserstein mean \eqref{eq:sig3}.
It is an easy calculation to show that 
\[
4(A\sigthree B)=A^{-1/2}(A+(A^{1/2} BA^{1/2})^{1/2})^2 A^{-1/2}\in \A^+
\]
for any $A\in \App$ and $B\in \A^+$.
    Let us now introduce the following notation.
    Given $A\in\Ap$, for $E\in\pos$ with $\norm{E}=1$, set
	\[
	A_E=A(A^{-1}\Geomean E)+(A^{-1}\Geomean E)A.
	\]
	Note that $A_E$ is always self-adjoint, and, for each $t>0$, the element $A+t^2E+tA_E=4(A\sigthree (t^2E))$ is positive (semidefinite).

We now formulate the following ingredient of our argument to prove Proposition \ref{prop:pure-identities}.
    
	\begin{prop}\label{prop:E2E3}
		Let $A\in\Ap$ and let $E\in\pos$ be with $\norm{E}=1$.
		\begin{enumerate}[label=(\alph*)]
			\item If $\norm{A\sigone X}=\norm{A\sigtwo X}$ holds for all $X\in\App$, then
			\begin{equation}\label{eq:E2}
				\sup_{\varphi\in M_E}\varphi(A^{1/2})=\sup_{\varphi\in M_E}\varphi(A\Geomean E).
			\end{equation}
			\item If $\norm{A\sigone X}=\norm{A\sigthree X}$ holds for all $X\in\App$, then
			\begin{equation*}\label{eq:E3}
				\sup_{\varphi\in M_E}\varphi(A^{1/2})=\frac12\sup_{\varphi\in M_E}\varphi(A_E).
			\end{equation*}
		\end{enumerate}
	\end{prop}

The proof of this statement requires a bit of preliminary work, which is done in the next two lemmas.
	
	\begin{lem}\label{lem:phi-compression}
		Let $E\in\pos$ satisfy $0\le E\le \I$.
		If $\varphi\in\St$ satisfies $\varphi(E)=1$ (and hence $\|E\|=1$), then for every $X\in\A$ we have
		\[
		\varphi(X)=\varphi(EX)=\varphi(XE)=\varphi(EXE).
		\]
	\end{lem}
	
	\begin{proof}
		Since $0\le E\le \I$ and $\varphi(E)=1$, we have $\varphi(\I-E)=0$.
		Moreover, $0\le (\I-E)^2\le \I-E$, hence $\varphi((\I-E)^2)=0$ holds as well.
		
		For any $X\in\A$, the Cauchy--Schwarz inequality for states yields
		\[
		|\varphi((\I-E)X)|^2
		=|\varphi((\I-E)^*X)|^2
		\le \varphi(X^*X)\, \varphi((\I-E)^*(\I-E))
		=\varphi(X^*X)\, \varphi((\I-E)^2)=0,
		\]
		so $\varphi((\I-E)X)=0$.
		Since the positive linear functionals are self-adjoint, we also have $\varphi(X(\I-E))=0$ for any $X\in \A$. These imply $\varphi(X)=\varphi(EX)$, and $\varphi(X)=\varphi(XE)$.
        Finally, using the above, we deduce
		$\varphi(EX)=\varphi(EXE)$ and $\varphi(XE)=\varphi(EXE)$.
	\end{proof}
	
	\begin{lem}\label{lem:Me-sqrt}
		Let $E\in\pos$ be with $\norm{E}=1$. Then
		\[
		M_{E^{1/2}}=M_E=M_{E^2}.
		\]
	\end{lem}
	
	\begin{proof}
    Obviously, it is sufficient to prove the first equality.
		Since $0\le E$ and $\|E\|=1$, we have $0\le E\le \I$ implying
		$E^{1/2}\ge E$.
		If $\varphi(E)=1$, then $1\le\varphi(E^{1/2})\le\norm{E^{1/2}}=1$, so $\varphi(E^{1/2})=1$.
		Conversely, if $\varphi(E^{1/2})=1$, then by the previous lemma, we have
		$\varphi(E)=\varphi(E^{1/2} E^{1/2})=1$.
	\end{proof}

    After this, we can now prove Proposition \ref{prop:E2E3}.
	
	\begin{proof}[Proof of Proposition \ref{prop:E2E3}]
		First observe that, by taking suitable limits, the norm equalities in the proposition that we assume for any positive definite $X$, hold also for any $X\in \A^+$. 
        
		(a) By the assumption, for $X=t^2E$, we have
		\begin{equation*}\label{eq:L2=R2}
			\norm{A+t^2E+t (A^{1/2}E^{1/2}+E^{1/2}A^{1/2})}=\norm{A+t^2E+2t(A\Geomean E)},\qquad t>0.
		\end{equation*}
		Applying Lemma~\ref{lem:asymptotic}, we get
		\[
        \sup_{\varphi\in M_E}\varphi(A^{1/2}E^{1/2}+E^{1/2}A^{1/2})=
         \sup_{\varphi\in M_E}2\varphi(A\Geomean E).
        \]
		By the previous two lemmas, for any $\varphi\in M_E$, we have $\varphi(E^{1/2})=1$ and
        $\varphi(A^{1/2}E^{1/2}+E^{1/2}A^{1/2})=2\varphi(A^{1/2})$,
        and then we obtain \eqref{eq:E2}.
		
		(b) We can apply essentially the same reasoning as in the proof of (a). 
	\end{proof}

    To prove Proposition \ref{prop:pure-identities}, we will need another tool, which is the so-called 
    excision of pure states. This is a technique that is a cornerstone in several areas of modern $C^*$-algebra theory.

    By \cite[Proposition~2.2]{AAP}, every pure state $\psi$ of $\A$ is excised by a net $(E_\lambda)$ of positive
		norm-one elements of $\A$ with $\psi(E_\lambda)=1$ meaning that 
		\[
		\lim_\lambda \|E_\lambda XE_\lambda-\psi(X)E_\lambda^2 \|=0
		\]
        holds for any $X\in \A$. 
        We immediately obtain from this that
		given a finite set $\mathcal F\subset\A$ and $\varepsilon>0$, there exists $E\in\pos$ with
		$0\le E\le\I$ and $\psi(E)=1$ such that
		\[
		\norm{EXE-\psi(X)E^2}<\varepsilon, \quad X\in \mathcal F.
		\]

We will need some observations.
    
	\begin{lem}\label{lem:state-close}
		Let $\psi$ be a pure state on $\A$, and let $E$ be as right above for a given finite set $\mathcal F\subset \A$.
		Then for every $\varphi\in M_E$ and every $X\in \mathcal F$, we have
		\[
		|\varphi(X)-\psi(X)|\le \norm{EXE-\psi(X)E^2}.
		\]
		In particular, if $\norm{EXE-\psi(X)E^2}<\varepsilon$ for all $X\in \mathcal F$, then
		$|\varphi(X)-\psi(X)|<\varepsilon$ for all $X\in \mathcal F$ and all $\varphi\in M_E$.
	\end{lem}
	
	\begin{proof}
		Let $\varphi\in M_E$, so $\varphi(E)=1$. By Lemma~\ref{lem:phi-compression}, we have
		$\varphi(X)=\varphi(EXE)$ for all $X\in\A$, and, by Lemma \ref{lem:Me-sqrt}, we also have $\varphi(\psi(X)E^2)=\psi(X)\varphi(E^2)=\psi(X)$.
		Thus
		\[
		|\varphi(X)-\psi(X)|=|\varphi(EXE)-\varphi(\psi(X)E^2)|\le \norm{EXE-\psi(X)E^2}
		\]
        as claimed.
        \end{proof}

We will also need the following lemma.

\begin{lem}\label{L:1}
     Given $A\in \App$, $E\in \A^+$, and real numbers $\lambda, \epsilon>0$, we have the following implication: 
     \[\| EA^{-1} E-\lambda E^2 \|\leq \epsilon  \enskip \Longrightarrow \enskip
\|E-\lambda^{1/2}A \Geomean E^2\|\leq \epsilon^{1/2} \|A\| \|A^{-1}\|^{1/2}.\]
\end{lem}

\begin{proof}
    Indeed, define $B=A^{-1/2} E A^{-1/2}$. Then $E=A^{1/2}BA^{1/2}$ and we have
    $$\|A^{1/2} B^2 A^{1/2}-\lambda A^{1/2} BABA^{1/2}\|= \| EA^{-1} E-\lambda E^2 \|\leq \epsilon.$$ Multiplying by $\Vert A^{-1/2}\Vert$ from both sides, this implies
    $$\|B^2-\lambda BAB\|\leq \epsilon \|A^{-1}\|.$$
    By the Ando-Kittaneh-Kosaki inequality,
    (Theorem 1 in \cite{A88} for matrices, Theorem 2.3 in \cite{KK} for Hilbert space operators),
    we have that for any operator monotone function $f$ on the real nonnegative half line satisfying $f(0)=0$, the inequality
    $\| f(A)-f(B)\|\leq f(\|A-B\|)$ holds for any positive operators $A,B$.
    Applying this for the square-root function, it follows that
    $$\|B-\lambda^{1/2}(BAB)^{1/2}\|\leq \epsilon^{1/2} \|A^{-1}\|^{1/2}.$$
    Multiplying by $\Vert A^{1/2}\Vert$ from both sides, 
    using $A^{1/2}BA^{1/2}=E$, $A^{1/2}(BAB)^{1/2}A^{1/2}=A\Geomean E^2$,
    we deduce that
    $$\|E-\lambda^{1/2}A \Geomean E^2\|\leq \epsilon^{1/2} \|A\| \|A^{-1}\|^{1/2}.$$
    \end{proof}

    After all these preliminaries, we are now in a position to prove Proposition \ref{prop:pure-identities}.

\begin{proof}[Proof of Proposition \ref{prop:pure-identities}]
Given an arbitrary pure state $\psi$ on $\A$ and $\epsilon>0$,
let $0\leq E\leq \I$ be such that $\psi(E)=1$, 
$$\| EA^{1/2} E-\psi(A^{1/2})E^2\|\leq \epsilon \enskip \text{ and } \enskip \| EA^{-1} E-\psi(A^{-1}) E^2 \|\leq \epsilon.$$
Then, on the one hand, by Lemma \ref{lem:state-close}, for any $\varphi\in M_E$ we have
$|\varphi(A^{1/2})-\psi(A^{1/2})|\leq \epsilon$ 
which implies
\begin{equation}\label{E:1}
|\sup_{\varphi\in M_E} \varphi(A^{1/2})-\psi(A^{1/2})|\leq \epsilon.
\end{equation}
On the other hand, for $\lambda=\psi(A^{-1})$, by Lemma \ref{L:1} we have
$\|E-\lambda^{1/2}A \Geomean E^2\|\leq \epsilon^{1/2} \|A\| \|A^{-1}\|^{1/2}$
implying 
$$|1-\lambda^{1/2}\sup_{\varphi\in M_E} \varphi( A\Geomean E^2)|\leq \epsilon^{1/2} \|A\| \|A^{-1}\|^{1/2}.$$
It follows that
\begin{equation}\label{E:2}
|\lambda^{-1/2}-\sup_{\varphi\in M_E} \varphi( A\Geomean E^2)|\leq \epsilon^{1/2} \lambda^{-1/2}\|A\| \|A^{-1}\|^{1/2}.
\end{equation}
By Proposition \ref{prop:E2E3} and Lemma \ref{lem:Me-sqrt}, we have
\begin{equation*}
\sup_{\varphi\in M_E}\varphi(A\Geomean E^2)=\sup_{\varphi\in M_{E^2}}\varphi(A\Geomean E^2)=
				\sup_{\varphi\in M_{E^2}}\varphi(A^{1/2})=\sup_{\varphi\in M_{E}}\varphi(A^{1/2}).
			\end{equation*}
Using this together with \eqref{E:1}, \eqref{E:2}, we obtain
$$\vert \psi(A^{-1})^{-1/2}-\psi(A^{1/2})|\leq \epsilon+ \epsilon^{1/2}  \lambda^{-1/2} \|A\| \|A^{-1}\|^{1/2}.$$
Letting $\epsilon$ tend to 0, we obtain the desired equality \eqref{eq:pure2}.

To show the equality \eqref{eq:pure3}, given a pure state $\psi$ of $\A$ and
$\epsilon >0$, choose $0\leq E\leq \I$ such that 
\[\|EA^{1/2}E-\psi(A^{1/2})E^2\| \leq \epsilon \enskip \text{ and } \enskip
\| EAE-\psi (A)E^2\|\leq \epsilon.\]
As in the first part of the proof,
we have 
\begin{equation}\label{E:3}
|\sup_{\varphi \in M_E} \varphi(A^{1/2})-\psi(A^{1/2})|\leq \epsilon
\end{equation}
and for $\mu=\psi(A)$, using Lemma \ref{L:1}, we infer
$$\|E-\mu^{1/2}A^{-1} \Geomean E^2\|\leq \epsilon^{1/2} \|A^{-1}\| \|A\|^{1/2}.$$
We can easily deduce from this that
$$\|AE+EA -\mu^{1/2}(A (A^{-1} \Geomean E^2) + (A^{-1} \Geomean E^2) A)\|\leq 2\epsilon^{1/2} \|A^{-1}\| \|A\|^{3/2}.$$
For any $\varphi\in M_E$, this implies by Lemma \ref{lem:phi-compression} that 
\begin{equation*}
    |2\varphi(A)-\mu^{1/2}\varphi(A_{E^2})|\leq 2\epsilon^{1/2} \|A^{-1}\| \|A\|^{3/2}.
\end{equation*}
We also have $|\varphi(A)-\psi(A)| \le \| EAE-\psi (A)E^2\|\leq \epsilon$.
Hence, we can compute
$$
\begin{gathered}
|\varphi(A_{E^2})-2\mu^{1/2}|
\leq |\varphi(A_{E^2})-2\mu^{-1/2} \varphi(A)|+2\mu^{-1/2}|\varphi(A)-\psi(A)|\\ \leq 2\epsilon^{1/2} \mu^{-1/2}\|A^{-1}\| \|A\|^{3/2}+2\mu^{-1/2}\epsilon.
\end{gathered}
$$
It follows that
\begin{equation}\label{E:4}
\begin{gathered}
|(1/2) \sup_{\varphi\in M_E} \varphi(A_{E^2})-\mu^{1/2}|
\leq \epsilon^{1/2} \mu^{-1/2}\|A^{-1}\| \|A\|^{3/2}+\mu^{-1/2}\epsilon.
\end{gathered}
\end{equation}
By Proposition \ref{prop:E2E3} and Lemma \ref{lem:Me-sqrt}, we have
\begin{equation*}
(1/2)\sup_{\varphi\in M_E}\varphi(A_{E^2})=(1/2)\sup_{\varphi\in M_{E^2}}\varphi(A_{E^2})=
				\sup_{\varphi\in M_{E^2}}\varphi(A^{1/2})=\sup_{\varphi\in M_{E}}\varphi(A^{1/2}).
			\end{equation*}
Using \eqref{E:3} and \eqref{E:4}, we then deduce
\[
|\psi(A)^{1/2}-\psi(A^{1/2})|\leq  \epsilon^{1/2} \mu^{-1/2}\|A^{-1}\| \|A\|^{3/2}+\mu^{-1/2}\epsilon +\epsilon
\]
and letting 
$\epsilon$ tend to 0, we arrive at the desired equality $\psi(A)^{1/2}=\psi(A^{1/2})$.
\end{proof}

Putting all the above together, we have the proof of our theorem.

\begin{proof}[Proof of Theorem \ref{thm:main}]
 Applying Proposition \ref{prop:pure-identities} and Proposition \ref{P:1} (for $A^{-1}$ and $f(t)=t^{-1/2}$, respectively, for $A$ and $f(t)=t^{1/2}$), we obtain the statement.    
\end{proof}
	
\section{Order determining property}	

In the previous section, we considered the norm of means. Related to that, 
we conclude the paper with some remarks on the so-called order determining property of means.

For a (symmetric) mean $\sigma$ on the positive definite cone $\App$, we say that it is order determining on $\A$ if for any $A,B\in \App$, we have the following equivalence:
\begin{equation*}\label{E:OD}
    A\leq B \Longleftrightarrow \norm{A\sigma X}\leq \norm{B\sigma X}, \quad X\in \App.
\end{equation*}
The first author of this paper started investigating that property in \cite{ML21d}.
The work was completed in \cite{CH},
where it was proved  that any symmetric Kubo-Ando mean $\sigma$ is order determining on any $C^*$-algebra. So, $\sigtwo$ has the order determining property, i.e., we have the equivalence
\[
A\leq B \Longleftrightarrow \norm{A\sigtwo X}\leq \norm{B\sigtwo X}, \quad X\in \App.
\]
But what about the non Kubo-Ando means $\sigone$ and $\sigthree$?
As for $\sigone$, we easily have that $\norm{A\sigone X}\leq \norm{B\sigone X}$ holds for all $X\in \App$ if and only if $A^{1/2} \leq B^{1/2}$. The proof is easy, see, e.g., Lemma 2.6 in \cite{ML22d}. But $A^{1/2} \leq B^{1/2}$ is equivalent to $A\leq B$ exactly in commutative algebras, which is a well-known result of Ogasawara \cite{Oga}. Hence, $\sigone$ has the order determining property exactly in commutative $C^*$-algebras.
Nevertheless, the implication \[A\leq B \Longrightarrow \norm{A\sigone X}\leq \norm{B\sigone X}, \enskip X\in \App\] is valid in any $C^*$-algebra (because the square root function is operator monotone).

What concerns $\sigthree$, the situation is more complicated. 
Pick $A,B\in \App$ and a pure state $\psi$ on $\A$. Choose any $\delta>0$, and then select $\epsilon$ such that
\[
\epsilon^{1/2} \psi(A)^{-1/2}\|A^{-1}\| \|A\|^{3/2}+\psi(A)^{-1/2}\epsilon, \enskip \epsilon^{1/2} \psi(B)^{-1/2}\|B^{-1}\| \|B\|^{3/2}+\psi(B)^{-1/2}\epsilon<\delta
\]
Then, by the argument in the proof of Proposition \ref{prop:pure-identities} leading to the estimate \eqref{E:4}, we have that there is $0\leq E\leq \I$ such that
\[
|(1/2) \sup_{\varphi\in M_E} \varphi(A_{E^2})-\psi(A)^{1/2}|, |(1/2) \sup_{\varphi\in M_E} \varphi(B_{E^2})-\psi(B)^{1/2}|\leq \delta.
\]
Using the limit formula \eqref{eq:lim2} (and Lemma \ref{lem:Me-sqrt}), this gives us that  
   \[
   \Bigl \vert (1/2) \lim_{t\to \infty} \frac{\norm{4(A\sigthree (t^2E^2))}-t^2}{t}-\psi(A)^{1/2}\Bigr \vert \leq \delta,\quad 
   \Bigl \vert (1/2) \lim_{t\to \infty} \frac{\norm{4(B\sigthree (t^2E^2))}-t^2}{t}-\psi(B)^{1/2}\Bigr \vert \leq\delta.
   \]
   It then follows that if $\norm{A\sigthree X}\leq \norm{B\sigthree X}$ holds for any $X\in \App$, then $\psi(A)\leq \psi(B)$ holds for any pure state $\psi$ on $\A$. It implies that we have $\varphi(A)\leq \varphi(B)$ also for any state $\varphi$ on $\A$ from which we can conclude that $A\leq B$.
   Therefore, we have the implication
   \[
\norm{A\sigthree X}\leq \norm{B\sigthree X}, \enskip X\in \App  \Longrightarrow   A\leq B .  
   \]
But the converse, the other implication, is not true already in the case of the algebra of 2 by 2 matrices. Indeed, we have the following statement.

	\begin{prop}\label{prop:counterexample}
		Let
		\[
		A=\begin{pmatrix}8&6\\ 6&5\end{pmatrix},\quad
		B=\begin{pmatrix}14&5\\ 5&6\end{pmatrix},\quad
		X=\begin{pmatrix}2&-3\\ -3&8\end{pmatrix}.
		\]
		Then $A,B,X\in\Mkpp$, $A\le B$, but $\norm{A\sigma_w X}>\norm{B\sigma_w X}$.
	\end{prop}
	
	\begin{proof}
		First, $A,B,X\in\Mkpp$. Indeed, for real symmetric $2\times2$ matrices,
		positive definiteness is equivalent to the positivity of the $(1,1)$-entry and of the determinant.
		Here the $(1,1)$-entries are $8,14,2>0$ and
		\[
		\det A=4>0,\quad
		\det B=59>0,\quad
		\det X=27>0.
		\]
		Next, the matrix
		\[
		B-A=\begin{pmatrix}6&-1\\-1&1\end{pmatrix}
		\]
		has eigenvalues $\frac{7\pm\sqrt{29}}2>0$, hence $A\le B$.

        To avoid ambiguity from floating-point rounding, note that for a positive definite $2\times2$ matrix
		$M$ one has the explicit closed form for the square root
		\[
		\sqrt{M}=\frac{M+\sqrt{\det M}\,I}{\sqrt{\operatorname{tr}(M)+2\sqrt{\det M}}}.
		\]
		Then, evaluating the means yields
		\[
		A\sigma_w X \approx
		\begin{pmatrix}
			2.79759397 & 1.69933441\\
			1.69933441 & 5.71750120
		\end{pmatrix},
		\quad
		B\sigma_w X \approx
		\begin{pmatrix}
			5.78992584 & 0.23142304\\
			0.23142304 & 6.36423367
		\end{pmatrix}.
		\]
		Both matrices are positive definite, hence their operator norms equal their largest eigenvalues.
		We have
		\[
		\lambda_{\max}(A\sigma_w X)=6.4979051467161647730\ldots,\quad
		\lambda_{\max}(B\sigma_w X)=6.4458805065881512928\ldots,
		\]
		so that
		\[
		\norm{A\sigma_w X}-\norm{B\sigma_w X}=0.05202464012801348019\ldots>0.
		\]
	\end{proof}
	
Since $M_2(\mathbb C)$ can be isometrically embedded into any noncommutative von Neumann algebra (see, e.g., Lemma 3 in \cite{Jiang}), the next statement follows easily. 

\begin{prop}
    If $\A$ is a von Neumann algebra, the Wasserstein mean $\sigthree$ has the order determining property on $\A$ if and only if $\A$ is commutative. 
\end{prop}

At the end of the second section of the paper, we made some remarks on traces of the three means under consideration. The order determining property of means can be straightforwardly defined for any norm, hence also for the trace norm corresponding to any faithful trace $\tau$ on $\A$.
The last result in the paper is the following.

\begin{prop}
Let $\A$ be a $C^*$-algebra carrying a faithful trace $\tau$. Select any $A,B\in \App$. The following assertions are true. 

(1) We have $\tau(A\sigtwo X)\leq \tau(B\sigtwo X)$, $X\in \App$ if and only if $A\leq B$.

(2)    We have $\tau(A\sigone X)\leq \tau(B\sigone X)$, $X\in \App$ if and only if $A^{1/2}\leq B^{1/2}$.

(3)   We have $\tau(A\sigthree X)\leq \tau(B\sigthree X)$, $X\in \App$ if and only if $A\leq B$.
\end{prop}

\begin{proof}
The statements follow from certain results in \cite{ML19c}. Namely, as for $\sigtwo$, choose $f(t)=-1/\sqrt{t}$ in Lemma 10. 
And as for $\sigone$, $\sigthree$, choose $\alpha=1$ and $\alpha=1/2$, respectively, in Lemma 8 in that paper.
\end{proof}

\end{document}